\documentclass{article}

\usepackage[utf8]{inputenc}
\usepackage[english]{babel}
\usepackage{amssymb,amsthm,amsmath,amsfonts,mathrsfs}
\usepackage[table]{xcolor}
\usepackage{tikz}
\usepackage{hyperref}
\usepackage{xfrac}
\usepackage{wrapfig}
\usepackage[normalem]{ulem}
\usepackage[capitalise, noabbrev]{cleveref}
	\crefname{subsection}{Subsection}{Subsections}
	\Crefname{subsection}{Subsection}{Subsections}

\newtheorem{theorem}{Theorem}[section]

\newtheorem{lemma}[theorem]{Lemma}

\newtheorem{corollary}[theorem]{Corollary}

\theoremstyle{definition}
\newtheorem{definition}[theorem]{Definition}
\newtheorem{example}[theorem]{Example}

\theoremstyle{remark}

\newtheorem{observation}[theorem]{Observation}

\renewcommand{\P}{\mathcal{P}}

\newcommand{\GL}{{G^\mathcal{L}}}
\newcommand{\GR}{{G^\mathcal{R}}}
\newcommand{\HL}{H^\mathcal{L}}
\newcommand{\HR}{H^\mathcal{R}}

\emergencystretch=\maxdimen
\begin{document}

\title{ Ordinal Sums, {\sc clockwise hackenbush}, and {\sc domino shave}}
\maketitle

\begin{center}
  {\sc Alda Carvalho}\textsuperscript{1},
  {\sc Melissa A. Huggan}\textsuperscript{2},
    {\sc Richard J. Nowakowski}\textsuperscript{3},\\
      {\sc Carlos Pereira dos Santos}\textsuperscript{4}
  \par \bigskip

  \textsuperscript{1}ISEL--IPL \& CEMAPRE--University of Lisbon, \href{mailto:acarvalho@adm.isel.pt}{acarvalho@adm.isel.pt} \par
  \textsuperscript{2}Ryerson University, \href{mailto:Melissa.Huggan@ryerson.ca}{melissa.huggan@ryerson.ca} \par
          \textsuperscript{3}Dalhousie University, \href{mailto:r.nowakowski@dal.ca}{r.nowakowski@dal.ca} \par
    \textsuperscript{4}Center for Functional Analysis, Linear Structures and Applications, University of Lisbon \& ISEL--IPL, \href{mailto:cmfsantos@fc.ul.pt}{cmfsantos@fc.ul.pt}

   \bigskip
Dedicated to Elwyn R. Berlekamp, John H. Conway and Richard K. Guy, they taught us so much.
\end{center}

\begin{abstract}
\noindent
We present two rulesets, \textsc{domino shave} and \textsc{clockwise hackenbush}. The first is somehow natural and, as
special cases,  includes \textsc{stirling shave} and Hetyei's Bernoulli game. \textsc{Clockwise hackenbush} seems artificial yet it is equivalent to \textsc{domino shave}. From the pictorial form of the game, and a knowledge of \textsc{hackenbush}, the decomposition into ordinal sums is immediate.  The values of \textsc{clockwise blue-red hackenbush} are numbers and we provide an explicit formula for the ordinal sum of numbers where the literal form of the base is $\{x\,|\,\}$ or $\{\,|\,x\}$, and $x$ is a number. That formula generalizes van Roode's signed binary number method for {\sc blue-red hackenbush}.
\end{abstract}

\section{Introduction}
\textsc{Hackenbush} is a central game in \textit{Winning Ways} \cite{BCG}. It has many interesting properties. One that will be central to this paper is the relationship between the ordinal sum decomposition and the valuation scheme for paths and trees. The literature also includes variants with new intriguing properties in new contexts. For example, \textsc{yellow-brown hackenbush} \cite{Berle2009} and all-small games;  \textsc{hackenbush sprigs} \cite{McKayMN2016} and mis\`ere games; and  \textsc{toppling dominoes} \cite{Fink} and hot games.
 
 In this paper, we introduce two rulesets, \textsc{clockwise hackenbush} and \textsc{domino shave}.
 The first is
 a new variant of
 \textsc{hackenbush} trees and the second is the partizan version of \textsc{stirling shave}.

 We first provide a complete solution for \textsc{clockwise blue-red hackenbush}.
 As in  \textsc{blue-red hackenbush} trees, the best moves are the ones highest up the tree, see \cref{lem:numbers}. We then give a method for calculating the value of a position.  This is accomplished by
 giving a decomposition theorem in term of ordinal sums (\cref{thm:OSDecomp}).  In  \cref{roode}
 explicit formulas are given for  the ordinal sum of numbers when the base is a {\sc blue-red hackenbush} string or when the base is in canonical form \cite{Roode}. In contrast, the evaluation of a tree in \textsc{blue-red hackenbush} involves iterating ordinal sums via signed binary numbers and disjunctive sums.

One of the main contributions of this paper is \cref{main2}, which gives the formula for the ordinal sum of numbers where the literal form of the base is $\{x\,|\,\}$ or $\{\,|\,x\}$, and $x$ is a number.

Whereas \textsc{clockwise hackenbush} may seem a little artificial, \textsc{domino shave} seems natural. It is
  a partizan version of \textsc{stirling shave} \cite{Fisher} which, in turn, was suggested by Hetyei's Bernoulli game \cite{Hetye2009,Hetye2010}.    The main result in \cref{sec:DominoShave} is that
  \textsc{clockwise hackenbush} and \textsc{domino shave} are equivalent games. Moreover,
  a position in one can be easily transformed to a position in the other. As an interesting sidelight, we also show that
Hetyei's Bernoulli game  is an instance of \textsc{stirling shave} thereby giving the first complete analysis of the game.

\subsection{The rules of the games}

A {\sc clockwise  hackenbush} position is a  \emph{tree} with blue, red, and green edges, which are connected to the ground. The rightmost edges form the \textit{trunk}, and the players can only remove edges from the trunk. There are two players, Left and Right. On Left's turn, she may remove a blue or green edge from the trunk. On Right's turn, he may remove a red or green edge from the trunk. Afterward, any edge not connected to the ground is also removed.

 We draw the trunk vertically. \cref{fig:trunk1} and \cref{fig:trunk2} show two {\sc clockwise blue-red hackenbush} positions and their options. Note that as play progresses, a branch that was not on the trunk can become part of the trunk. See the first Left option in \cref{fig:trunk1} and \cref{fig:trunk2}. Different drawings of the same tree will result in different trunks and therefore in different
  \textsc{clockwise hackenbush} positions.

\begin{figure}[hbt]
\begin{center}
\includegraphics[width=0.8\linewidth]{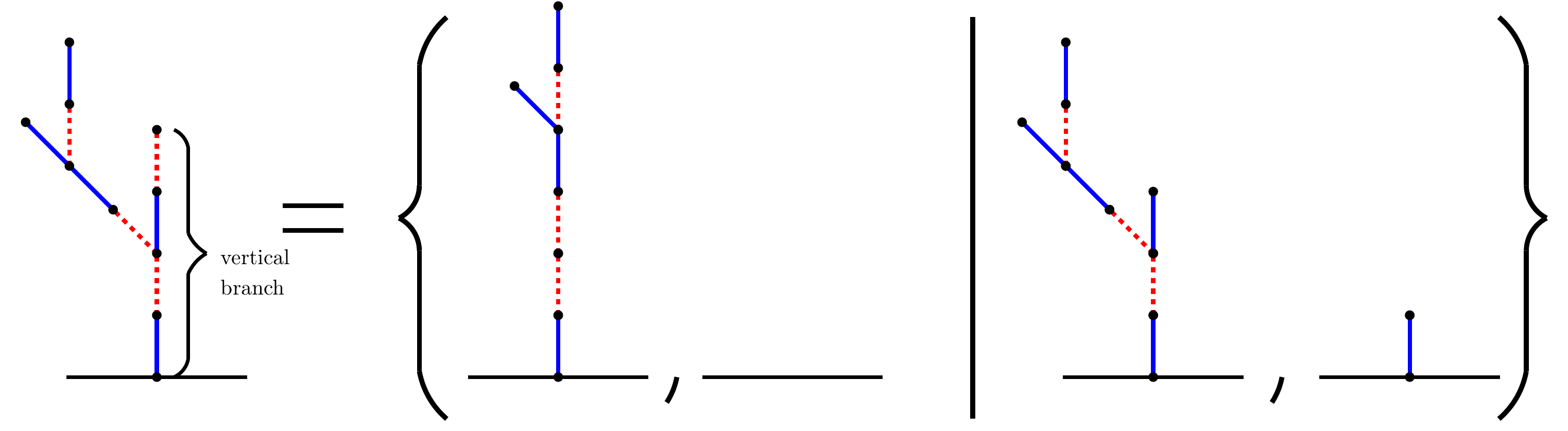}
\caption{A \textsc{clockwise blue-red hackenbush} position.}\label{fig:trunk1}
\end{center}
\end{figure}

\begin{figure}[hbt]
\begin{center}
\includegraphics[width=0.8\linewidth]{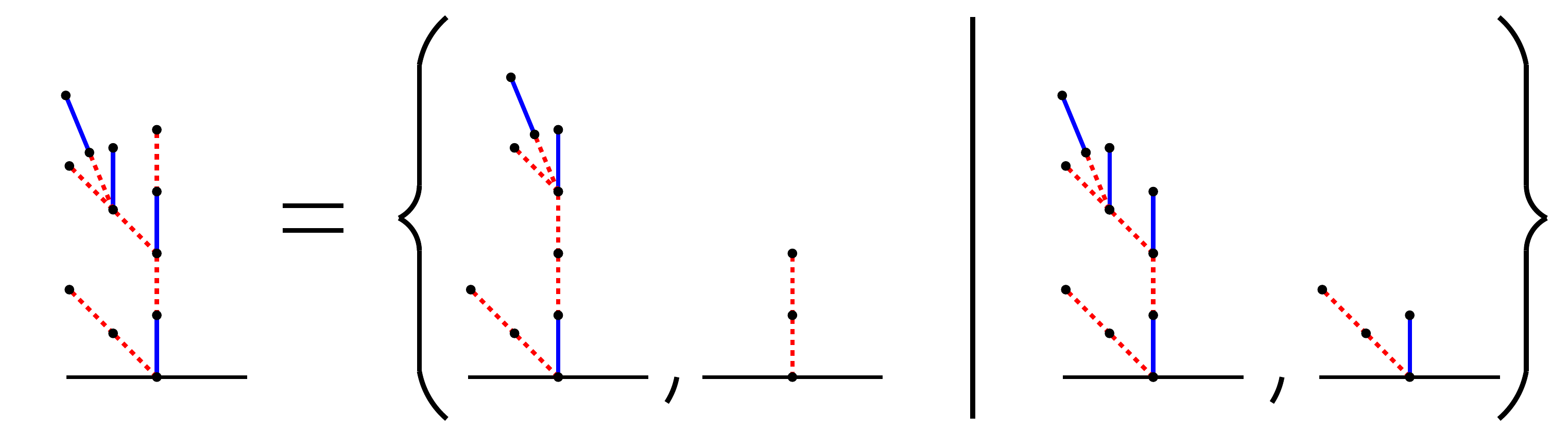}
\end{center}
\caption{A second \textsc{clockwise blue-red hackenbush} position.}\label{fig:trunk2}
\end{figure}

 \textsc{Domino shave}, not surprisingly, involves dominoes. For us, a \textit{domino} is an ordered pair of non-negative integers,
 written $d=(l,r)$. We will distinguish the numbers: $l$ is the \textit{left spot} and $r$ is the \textit{right spot}. A line of $k$ dominoes will be described as $d_1,d_2,\ldots,d_k$ or as $(l_1,r_1),(l_2,r_2),\ldots,(l_k,r_k)$, $0\leqslant l_i, r_i$. A domino is  \textit{blue}  if $l_i<r_i$, it is  \textit{red}  if $l_i>r_i$, and  \textit{green} if $l_i=r_i$.

A {\sc domino shave} position is a line of dominoes. The two players take turns making moves. On Left's move, she may remove a green or blue domino $d_i$ and all the others with greater index  leaving $d_1,d_2,\ldots, d_{i-1}$, provided that, for all $j\geqslant i$, $l_i\leqslant  l_j$ and $l_i\leqslant r_j$. On Right's move, he may remove a green or red domino $d_i$ leaving $d_1,d_2,\ldots, d_{i-1}$, provided that, for all $j\geqslant i$, both $l_j\geqslant r_i$ and $r_j\geqslant r_i$ hold. See Figure~\ref{fig: domino} for an example of a {\sc domino shave}  position and its options.\\

\begin{figure}[hbt]
\begin{center}
\includegraphics[width=0.9\linewidth]{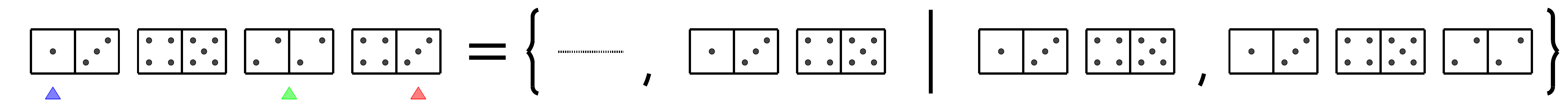}
\caption{Example of a \textsc{domino shave} position.}\label{fig: domino}
\end{center}
\end{figure}

For this paper, normal play is the winning convention.  Readers can consult any edition of \textit{Winning Ways}~\cite{BCG}, specifically the sections on \textsc{hackenbush}, to gain further insight. We assume general knowledge about normal play but, in order to keep the material self-contained, we clarify some ideas about the concepts of ordinal sum and, also, the particular case of ordinal sums of \linebreak{\sc blue-red hackenbush} strings.

\subsection{Ordinal sum}

In a {\sc blue-red hackenbush} string, if a player moves on the bottom, then the top disappears; if a player moves on the top, then nothing happens to the bottom. This idea motivates the concept of the ordinal sum. In the ordinal sum of two games $G:H$, a player may move in either $G$ (base) or $H$ (subordinate), with the additional constraint that any move on $G$ completely annihilates the component $H$. The recursive definition is $$G:H=\left\{G^\mathcal{L},G:H^\mathcal{L}\,|\,G^\mathcal{R},G:H^\mathcal{R}\right\}.$$
The Colon principle states that the form of the base matters, but not the form of the subordinate. Formally, \\

\noindent \textit{Colon Principle} \cite{BCG}: If $H\geqslant H'$, then $G:H\geqslant G:H'$.\\

 Note that, while it is true that $H=H'$ implies $G:H=G:H'$, it is \textit{not} true that $G=G'$ implies $G:H=G':H$. For example, $G=\{0\,|\,2\}$ and $G'=\{0\,|\,\}$ are different forms with game value $1$, and we have $G:1=1\frac{1}{2}$ and $G':1=2$.\\

In fact, we can be more precise about the role of the base of an ordinal sum. The following theorem shows that the problem only happens if the literal form of the base has reversible options. If it has no reversible options, we can replace the literal form of the base by its canonical form without changing the game value.

\begin{theorem}[McKay's Theorem]\label{th:mk} If $G$ has no reversible options and $K$ is the canonical form of $G$, then $G:H=K:H$.
\end{theorem}

\begin{proof}
See \cite{Mc}, page 42.
\end{proof}

Here, we will prove some results about ordinal sums with the form $\{G\,|\,\}:H$. In those ordinal sums, $G$ is not the base; the base is $\{G\,|\,\}$. Also, as stated in Theorem \ref{th:g}, the game value of $\{G\,|\,\}:H$ does not depend on the game form of $G$.

\begin{theorem}\label{th:g} Let $G$, $G'$ and $H$ be game forms. If $G=G'$ then\linebreak $\{G\,|\,\}:H=\{G'\,|\,\}:H$.
\end{theorem}

\begin{proof}
Suppose that, in the game $\{G\,|\,\}:H+\{\,|\,-G'\}:(-H)$, Right moves to $\{G\,|\,\}:H^R+\{\,|\,-G'\}:(-H)$ or to $\{G\,|\,\}:H+\{\,|\,-G'\}:(-H^L)$. Then, Left answers $\{G\,|\,\}:H^R+\{\,|\,-G'\}:(-H^R)$ or $\{G\,|\,\}:H^L+\{\,|\,-G'\}:(-H^L)$ respectively, and, by induction, she wins. On the other hand, if Right moves to $\{G\,|\,\}:H-G'$, Left replies $G-G'$ and wins, since $G=G'$. Analogously, if Left plays first in  $\{G\,|\,\}:H+\{\,|\,-G'\}:(-H)$, then she loses. Hence,  $\{G\,|\,\}:H+\{\,|\,-G'\}:(-H)$ is a $\mathcal{P}$-position and $\{G\,|\,\}:H=\{G'\,|\,\}:H$.
\end{proof}

\subsection{Ordinal sums of {\sc blue-red hackenbush} strings}

It is known that the game values of {\sc blue-red hackenbush} strings are numbers and that there is a correspondence between the game values of {\sc blue-red hackenbush} strings and \emph{signed binary representations} \cite{Roode}.

{The part of a {\sc blue-red hackenbush} string after the first color change is represented by the digits after the binary point, whose value is a sum of powers of $2$. So, when we write $n.\overline{1}\,\overline{1}1\ldots$ (the overlines indicate negative powers of $2$); the represented value is $n-\frac{1}{2}-\frac{1}{4}+\frac{1}{8}+\ldots$ The signed binary notation is particulary appropriate for simultaneously describing the game value of the {\sc blue-red hackenbush} string and its sequence of blue and red edges. In the following example, $2.\overline{1}\,\overline{1}1$ stands for two blue edges, one red edge, one red edge, and one blue edge; see Figure~\ref{Fig: example of game}.

\begin{figure}[ht]
\begin{center}
\scalebox{0.8}{
\includegraphics[width=\linewidth]{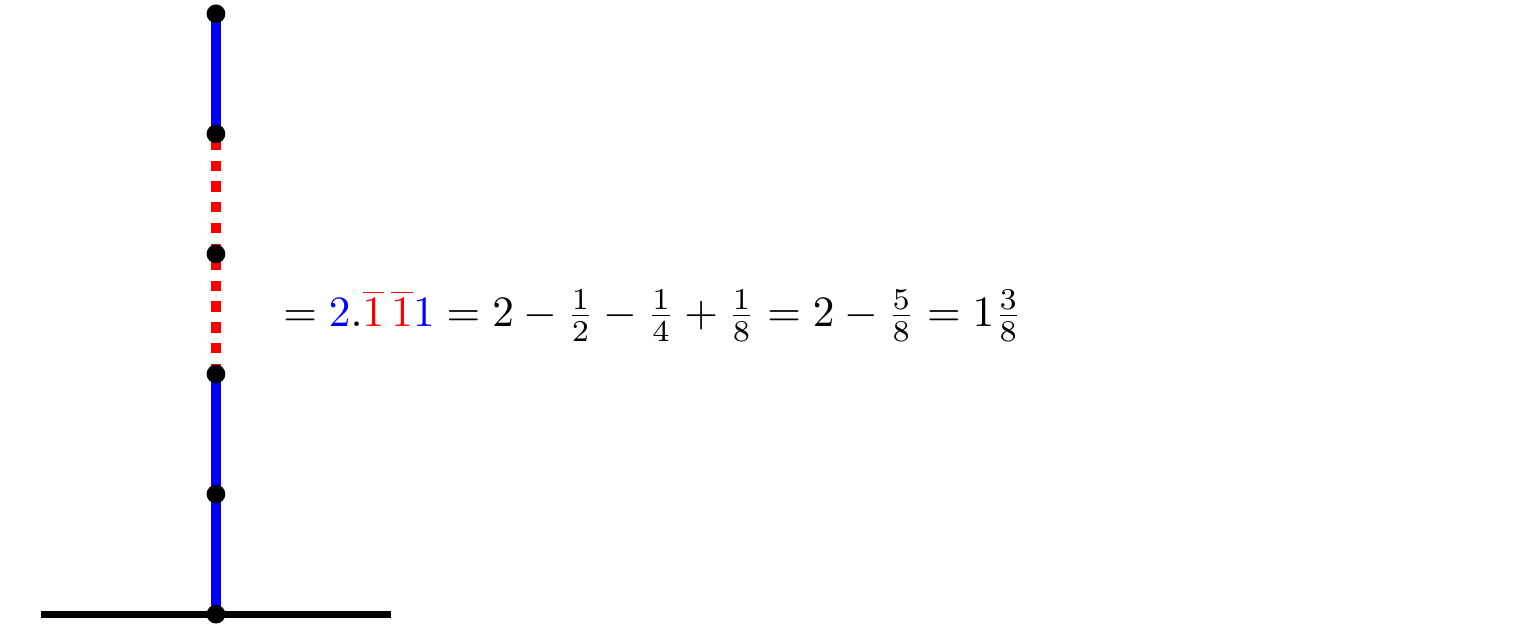}
}
\caption{Example of signed binary notation for a {\sc blue-red hackenbush} string.}\label{Fig: example of game}
\end{center}
\end{figure}
Also, if $G$ and $H$ are two {\sc blue-red hackenbush} strings, it is possible to have a closed formula to evaluate the game value of $G:H$, knowing the game values of $G$ and $H$. That is van Roode's method \cite{Roode}.

\begin{theorem}[van Roode's method]\label{roode}
Let $G$ be a positive  {\sc blue-red hackenbush} string whose game value is $n+d$, with $-1<d=-\frac{k}{2^j}\leqslant 0$. Then, we have the following.
\begin{enumerate}
  \item If $G$ is an integer and $H$ is a positive  {\sc blue-red hackenbush} string then $G:H=G+H$.
  \item If $G$ is not an integer and $H$ is a positive  {\sc blue-red hackenbush} string whose game value is $m+d'$, where $m$ is a positive integer and $-1<d'\leqslant 0$, then $G:H=n+d+\frac{1}{2^{j+m}}\left(2^m-1+d'\right)$.
  \item If $H$ is a negative  {\sc blue-red hackenbush} string whose game value is $m+d'$, where $m$ is a negative integer and $0\leqslant d'< 1$ then \linebreak $G:H=n+d+\frac{1}{2^{j+|m|}}\left(1-2^{|m|}+d'\right)$.
\end{enumerate}
\end{theorem}

\begin{proof} This result is well known and follows from either Berlekamp's or van
Roode's rule for a {\sc blue-red hackenbush} string \cite{Roode} and \cite{Ber}.
\end{proof}

\begin{example}
Consider the games $G$ and $H$, as follows:

\begin{minipage}[t]{0.5\textwidth}
\begin{center}$G=$ \raisebox{-1cm}{\includegraphics[scale=0.5]{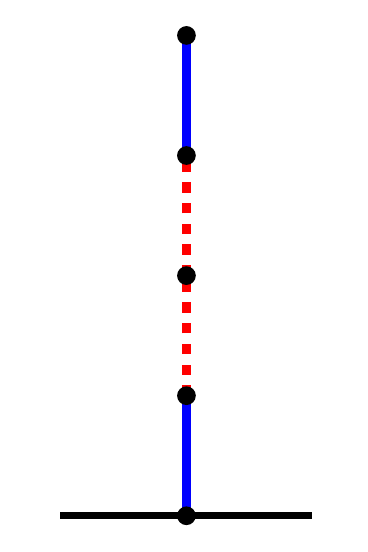}} $=\frac{3}{8}=1-\frac{5}{8}$\end{center}

$$n=1,\,d=-\frac{5}{2^3},\,j=3$$
\end{minipage}
\quad
\begin{minipage}[t]{0.4\textwidth}

\begin{center}$H=$ \raisebox{-1cm}{\includegraphics[scale=0.5]{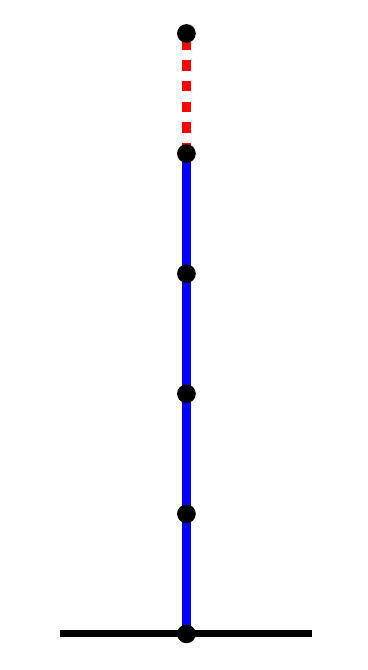}} $=3\frac{1}{2}=4-\frac{1}{2}$\end{center}

$$m=4,\,d'=-\frac{1}{2}$$
\end{minipage}

\vspace{1cm}
We want to evaluate $G:H$,

\begin{center}$G:H=$ \raisebox{-1cm}{\includegraphics[scale=0.5]{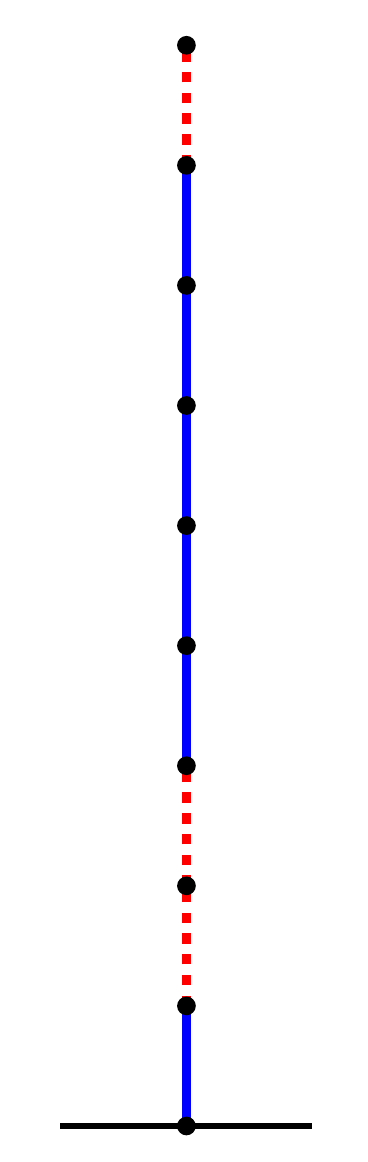}}.\end{center}

We have
$$G:H=n+d+\frac{1}{2^{j+m}}\left(2^m-1+d'\right)=1-\frac{5}{8}+\frac{1}{2^7}\left(2^4-1-\frac{1}{2}\right)=\frac{125}{256}.$$
\end{example}

\begin{example}Consider the games $G$ and $H$, as follows:

\begin{minipage}[t]{0.4\textwidth}
\begin{center}$G=$ \raisebox{-1cm}{\includegraphics[scale=0.5]{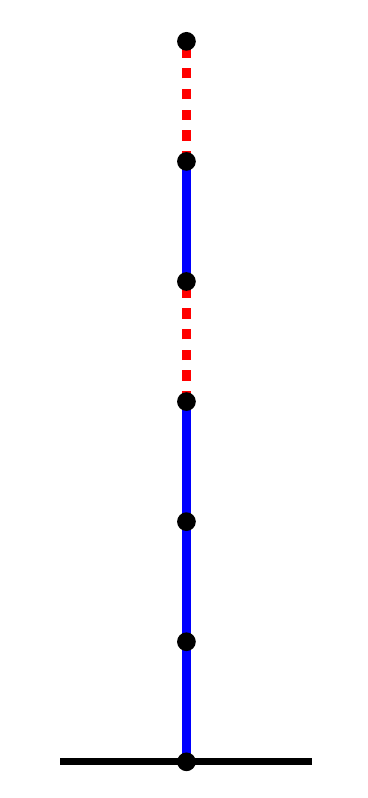}} $=2\frac{5}{8}=3-\frac{3}{8}$\end{center}

$$n=3,\,d=-\frac{3}{2^3},\,j=3$$
\end{minipage}
\quad
\begin{minipage}[t]{0.5\textwidth}
\begin{center}$H=$ \raisebox{-1cm}{\includegraphics[scale=0.5]{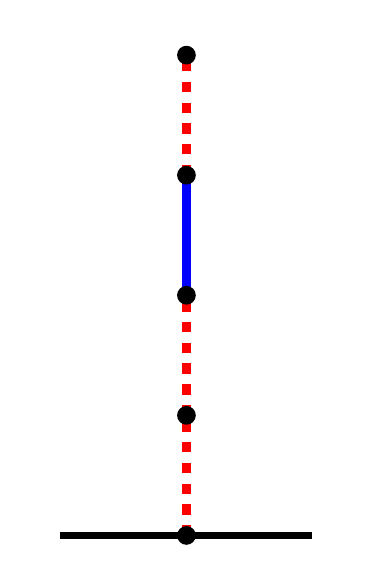}} $=-1\frac{3}{4}=-2+\frac{1}{4}$\end{center}

$$m=-2,\,d'=\frac{1}{4}$$
\end{minipage}

\vspace{1cm}
We want to evaluate $G:H$,

\begin{center}$G:H=$ \raisebox{-1cm}{\includegraphics[scale=0.5]{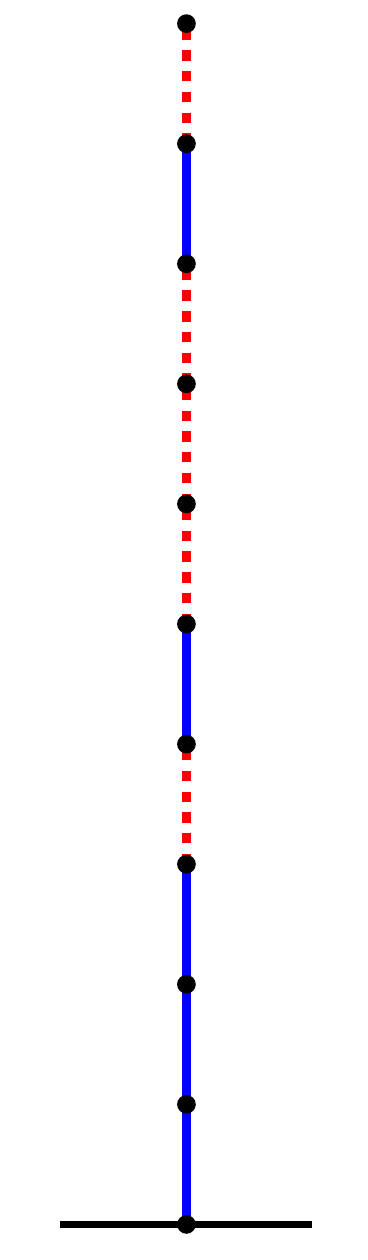}}.\end{center}

We have
$$G:H=n+d+\frac{1}{2^{j+|m|}}\left(1-2^{|m|}+d'\right)=3-\frac{3}{8}+\frac{1}{2^5}\left(1-2^2+\frac{1}{4}\right)=2\frac{69}{128}.$$
\end{example}

Van Roode's method was conceived to evaluate ordinal sums of {\sc blue-red hackenbush} strings. However, since {\sc blue-red hackenbush} strings have no reversible options, by Theorem \ref{th:mk}, this method can be used to evaluate an ordinal sum of numbers where the base is in canonical form.

\section{The analysis of {\sc clockwise  hackenbush}}\label{sec:ch}

In order to analyse {\sc clockwise  hackenbush} positions and facilitate the proofs, it is important to have notation for the important elements.

\begin{definition}
Let $G$ be a {\sc clockwise blue-red hackenbush} position. Let $T_G$ be the trunk of $G$ with  $V(T_G)=\{s_0,s_1,\ldots,s_n\}$
and $E(T_G)=\{t_1,t_2,\ldots,t_n\}$, all labelled from bottom to top. Let $G_i$ be the position resulting from the deletion of $t_i$ and let $S_1=G_1$ and for $i>1$, $S_i=G_i\setminus (G_{i-1}\cup \{t_{i-1}\})$. Finally, for $i \geqslant 1$, let $M_i=S_i\cup\{t_i\}$.
\end{definition}

The subtree $S_1$ is the part of the tree remaining after deleting $t_1$ and,  for $i>1$, not counting with $t_i$, $S_i$ is the part of the tree that is eliminated by deleting $t_{i-1}$ but not by deleting $t_i$. In other words, the subtree above $t_{i-1}$ that does not include $t_i$. The idea is represented in Figure \ref{fig: notation}.

\begin{figure}[hbt]
\begin{center}
\includegraphics[width=0.5\linewidth]{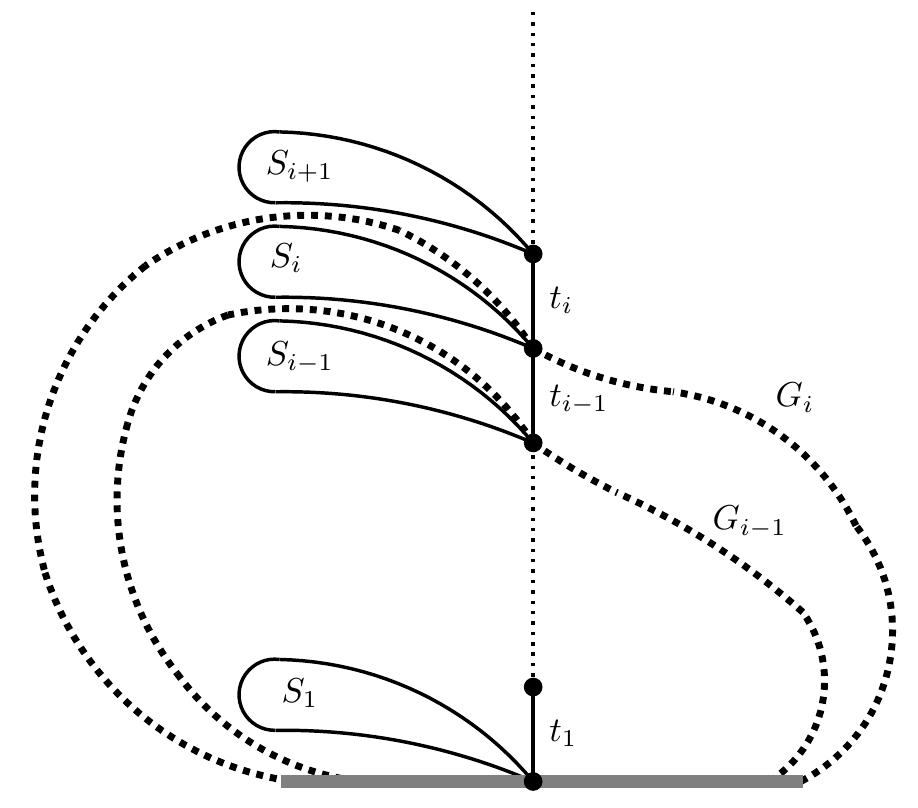}
\caption{Notation for the elements of a {\sc clockwise  hackenbush} position.}\label{fig: notation}
\end{center}
\end{figure}

\begin{theorem}\label{thm:OSDecomp} Let $G$ be a {\sc clockwise blue-red hackenbush} position. Then
$G=M_1:(M_2:(\ldots:(M_{n-1}:M_n)\ldots))$.
\end{theorem}
\begin{proof}We induct on  the size of $G$. If $E(T_{G})=\{t_1\}$ then $G=M_1$.


We may now suppose  that $E(T_G)=\{t_1,t_2,\ldots,t_n\}$ and $n>1$. Let $H$ be the position formed by $G\setminus M_1$, that is, the tree above but not including $t_1$, and the vertex $s_1$ is the ground. The trunk of $H$ is $\{t_2,t_3,\ldots, t_n\}$.

In $G$, there are two types of moves. Either, in $M_1$, delete $t_1$ leaving $S_1$; or delete $t_i$, $i>1$ which is a move in  $H$. By induction,
the move in $H$ is to $M_1:H^L$ ($M_1:H^R$) for some Left (Right) option of $H$. Also by induction, $H=M_2:(\ldots:(M_{n-1}:M_n)\ldots)$. It follows then that
\begin{eqnarray*}
G&=&\{M_1^L, M_1:\HL \mid M_1^R, M_1:\HR \}\\
&=&M_1:(M_2:(\ldots:(M_{n-1}:M_n)\ldots)),
\end{eqnarray*}
 and the result is proved.
\end{proof}

\cref{thm:OSDecomp} shows that we will have to evaluate ordinal sums. If the values were arbitrary  then no formula could be given. However,  {\sc clockwise blue-red hackenbush} positions have similar strategic features to {\sc  blue-red hackenbush} strings. Specifically, for either player, the
unique best move is their highest and the value is a number. This we prove next. Each $M_i$ has only one option, that of deleting the trunk edge. The ordinal sums, therefore, will be of the form $\{x \, \mid \, \}:y$ or $\{\,\mid \,x\,\}:y$ for numbers $x$ and $y$. A closed formula for this type of ordinal sums is one of the main contributions of this paper (Subsection \ref{subsec:emptybase}). Before that, we prove that {\sc clockwise blue-red hackenbush} positions only have numbers as game values, and that the best options for the players are the topmost allowed moves.

\begin{lemma}\label{lem:numbers}
Let $G$ be a {\sc clockwise blue-red hackenbush} position.
If $t_i$ and $t_j$ are blue edges and  $j>i$, then $G_j>G_i$. If $t_i$ and $t_j$ are red edges and $i<j$, then $G_j<G_i$.
\end{lemma}

\begin{proof}
We first assume that $t_i$ and $t_j$ are both blue edges and we show  $G_j-G_i>0$.

We will induct on the number of edges in $G$. If $G$ consists of exactly two blue edges, then $G=2$.
If Left deletes the higher edge this leaves a tree with exactly one blue edge which has value $1$. If she deletes
 the lower edge this leaves a tree with zero edges and it has value $0$. Thus the lemma holds for the base case.
 We now suppose $G$ has more than two edges.

Left, going first, can win by deleting $t_i$ in $G_j$ since this results in $G_i-G_i=0$. \\
Now consider Right moving first. If Right plays an edge of  $G_j$ but does not eliminate the edge $t_i$ then Left responds in $G_j$ by deleting $t_i$.  Again, this results in $G_i-G_i = 0$. If Right plays in $G_j$ and does eliminate $t_i$ then he has deleted an edge on the trunk, i.e., some $t_\ell$, $\ell<i$. This leaves $G_{\ell}-G_i$. Left responds in $-G_{i}$ by deleting $t_\ell$, that, by symmetry, is a blue edge. This gives $G_\ell-G_\ell=0$.

The last remaining case is that Right deletes an edge on the new trunk in $-G_i$. Let the trunk of $G_i$  be
$T_1=\{t'_1,t'_2,\ldots,t'_m\}$ where $t'_a=t_a$ for $1\leqslant a\leqslant i-1$. Right deletes $t'_\ell$ for $i\leqslant \ell\leqslant m$.
We claim that deleting $t_i$, in $G_j$, is a winning move.  To see this, let $H$ be identical to $G_i$ but with an extra blue edge $T'_{m+1}$ at the top of $T_1$.  After Right has deleted $t'_\ell$ in $-G_i$ and Left $t_i$ in $G_j$, the situation is identical to playing in $H_{m+1}-H_\ell$. Now both $t_i$ and $t_j$ are not in $H$ thus $H$ has at least one fewer edge than $G$. It follows by induction that
$H_{m+1}-H_\ell>0$.

The proof for when $t_i$ and $t_j$ are both red edges is similar and omitted.
\end{proof}

\begin{corollary}\label{cor:numbers2} Let $G$ be a {\sc clockwise blue-red hackenbush} position.

\begin{enumerate}
\item Left's (Right's) move of deleting the topmost blue (red)  edge on the trunk dominates all other options.
\item The value of $G$ is a number.
\end{enumerate}
\end{corollary}

\begin{proof} Part 1 follows immediately from 
 \cref{lem:numbers}. Part 1 gives that $G$ has only one Left and one Right un-dominated option, i.e., $G=\{ G^L\mid G^R\}$. By induction on the options, both options $G^L$ and $G^R$  are numbers. Let $H$ be $G$ with an extra blue edge on the top of the trunk. Both $G$ and $G^L$ are
Left options of $H$ and, by \cref{lem:numbers}, $G^L<G$. Similarly, by adding a red edge, we have $G<G^R$. Thus $G$ is a number.
\end{proof}

\subsection{Simplicity rule and binary notation}

Iterated ordinal sums occur naturally in \textsc{clockwise hackenbush} and the goal of this section is to find a procedure that evaluates them.  In what follows, recall that the form of the base is important.  For example, let $n$ be a number and consider the ordinal sum $\{n\mid\,\}:2$. The good moves are the topmost, thus
$$\{\,n\,\mid\,\}:2=\{\{\,n\mid\,\}:1\,\mid\,\}=\{\{\{\,n\mid\,\}\,\mid\,\}\,\mid\,\}.$$
Similarly,
$$\{\,n\mid\,\}:-2=\{\,n\mid\,\{\,n\mid\,\}:-1\}=\{\,n\mid\,\{\,n\mid\,\{n\,\mid\,\}\}\}.$$ In either case, the Simplicity Rule must be applied three times in a row and one of the options remains the same. This motivates the following definition.

If $a$ and $b$ are numbers and $a<b$ then the value of $\{a\mid b\}$ is the dyadic rational $p/2^q$ with $a<p/2^q<b$
and $q$ is minimal. In other words, and a fact that we will use often:

\textit{$\{a\mid b\}$ is the number $c$, $a<c<b$, that has the fewest number of digits in its binary expansion.}

The next result makes explicit the simplicity rule for evaluating $\{a\mid b\}$ for numbers $0\leqslant a<1$ and $a<b$.
We will then generalize the rule for iterated ordinal sums in Section \ref{subsec:emptybase}.
The procedure will use the binary expansions of numbers.
Each dyadic only has a finite number of non-zero bits in its binary expansion,
however, the procedure sometimes uses 0-bits past the last 1-bit. Therefore, while we denote the binary expansion of $d$ by  $d=_20.d_1d_2\ldots d_n$ but when we refer to the `first index' or `first occurrence' we may be considering the infinite binary expansion.
We abuse the `$=_2$' notation to mean that the important terms following the equal sign will be in binary. If this is followed by another `$=$' sign then we have reverted to base 10.\\

\begin{theorem}\label{simplicity1}
Let  $d$ be a dyadic rational such that $0<d<1$, and $d=_20.d_1d_2\ldots d_n$. Let $d<d'\leqslant +\infty$, and, if $d'<1$, then let $d'=_20.d'_1d'_2\ldots d'_m$.

\begin{enumerate}
  \item If $d'>1$, then $\{d\mid d'\}=1$.
  \item If $d'=1$ and $i$ be the index of the first $0$-bit of the binary expansion of $d$,
  then $\{d\mid 1\}=1-\frac{1}{2^i}$.
  \item If $d'<1$, then let $i$ be the first index such that $d_i=0$ and $d'_i=1$. Also, let $j$ be the least index, $j>i$ and $d_j=0$.

      If $d'\neq_2 0.d_1d_2d_3\ldots d_{i-1}1$, then $\{d\mid d'\}=_20.d_1d_2d_3\ldots d_{i-1}1$.

      If $d'=_2 0.d_1d_2d_3\ldots d_{i-1}1$, then $\{d\mid d'\}=_20.d_1d_2d_3\ldots d_{j-1}1$.
\end{enumerate}
\end{theorem}

\begin{proof} The first case is trivial.\\

In the second case, by definition $\{d\mid 1\}>d$, then each of the first $i-1$ digits of the binary expansion of $\{d\mid 1\}$ must be ones.
 Therefore $\{d\mid 1\}\geqslant \frac{k}{2^j}$ with $j\geqslant i$. Observe now that inserting one more ``1'' in the position $i$ produces a dyadic strictly larger than $d$ and strictly smaller than 1. Therefore, the simplest dyadic that fits between $d$ and $1$ is $1-\frac{1}{2^i}=_20.11\ldots11$. \\

Regarding the third case, since $d<\{d\mid d'\}<d'$ the first $i-1$ digits of the binary expansion of $\{d\mid d'\}$
must be $d_1,d_2,d_3,\ldots,d_{i-1}$.
 Therefore, $\{d\mid d'\}=\frac{k}{2^w}$ for some $k$ and $w\geqslant i$. If $d'\neq_2 0.d_1d_2d_3\ldots d_{i-1}1$ then
 $0.d_1d_2d_3\ldots d_{i-1}1$ is the simplest dyadic that fits between $d$ and $d'$. If $d'=_2 0.d_1d_2d_3\ldots d_{i-1}1$, and since $d< \{d\mid d'\}<d'$, then the first $j-1$ digits of the binary expansion of $\{d\mid d'\}$ must be $d_1,d_2,d_3,\ldots,d_{j-1}$.  In that case, the simplest dyadic that fits between $d$ and $d'$ is
 $0.d_1d_2d_3\ldots d_{j-1}1$.
\end{proof}

\begin{example}
\begin{eqnarray*}
\left\{\frac{21}{32}\,\Big|\,\frac{45}{64}\right\}&=_2&\{0.10101\mid 0.101101\}=_20.1011=\frac{11}{16};\\
\left\{\frac{75}{128}\,\Big|\, \frac{19}{32}\right\}&=_2&\{0.1001011\mid 0.10011\}=_20.10010111=\frac{151}{256}.
\end{eqnarray*}
\end{example}

We have seen that there are two types of ordinal sums that occur in \textsc{clockwise blue-red hackenbush}. We write the formulas explicitly. The first, in \cref{roode}, is standard and appears in the analysis of \textsc{blue-red hackenbush} strings. The second happens when the literal form of the base is $\{x\,|\,\}$ or  $\{\,|\,x\}$, where $x$ is a number. That is analysed in the next section.

\subsection{Ordinal sums of numbers: the literal form of the base is $\{x\,|\,\}$ or  $\{\,|\,x\}$, where $x$ is a number} \label{subsec:emptybase}
The second type of ordinal sum that occurs in \textsc{clockwise blue-red hackenbush} is  $\{d\mid\, \}:m$. It still involves numbers but the base is not in canonical form. Some preliminary results are needed first.

If $n$ is a number then the Translation Principle states $\{\GL+n\,\mid\,\GR+n\}=n+\{\GL\,\mid\,\GR\}$ \cite{ANW,BCG,Con,Sie}. The following theorem describes a version of the translation principle for ordinal sums. Once we have this result,
the case $\{\,d\mid\,\}:\,$number $(0\leqslant d<1)$ turns out to be the only case to study.\\

\begin{lemma}[Translation principle for ordinal sums of numbers]\label{translation}
  Let $0\leqslant d< 1$ be a dyadic rational, $w$ any number, and $n$ an integer. Now,
   $$\{n+d\,\mid\,\}:w=n+\left(\{d\,\mid\,\}:w\right).$$
\end{lemma}

\begin{proof}

Let $\{w^L\,|\,w^R\}$ be the canonical form of $w$.
We have
\begin{eqnarray*}
   && \{n+d\,|\,\}:\{w^L\,|\,w^R\} \\
   &=& \left\{n+d,\{n+d\,|\,\}:w^L\,\Big|\,\{n+d\,|\,\}:w^R\right\}\\
   &\underbrace{=}_{\text{induction}}& \left\{n+d,n+\left(\{d\,|\,\}:w^L\right)\,\Big|\,n+\left(\{d\,|\,\}:w^R\right)\right\}\\
      &\underbrace{=}_{\text{translation principle}}& n+\left\{d,\{d\,|\,\}:w^L\,\Big|\,\{d\,|\,\}:w^R\right\}\\
       &=& n+\left(\{d\,|\,\}:\{w^L\,|\,w^R\}\right).
\end{eqnarray*}
\end{proof}
\begin{lemma}\label{lem:translation} Let  $d$ be a dyadic rational, $0\leqslant d <1$ and $m$ and integer.
If  $m>0$, then
$$\{d\mid\} :  m=\{\{d\mid\} :  m-1\mid \}.$$ If $m<0$, then $$\{d\mid\} :  m=\{d\mid \{d\mid \}:(m+1)\}.$$
\end{lemma}
\begin{proof} Let $m$ be a positive integer.  By definition,
\begin{eqnarray*}
\{d\mid\} :  m &=& \{d, \{d\mid\} :  0, \{d\mid\} :  1,\ldots, \{d\mid\} :  (m-1)\,\mid\,\},\\
\{d\mid\} :  (-m) &=& \{d\, \mid \, \{d \mid \} :  0, \{d \mid \} :  -1,\ldots, \{d \mid \} :  (-m+1)\}.
\end{eqnarray*}

For any integer $k$, let  $G=\{d\mid\} :  k - \{d\mid\} :  (k-1)$. We claim that $G\geqslant 0$. Suppose $k>0$.  In $G$,
Right can only play in $-\{d\mid\} :  (k-1)$ and for any move he makes, Left has the corresponding move in $\{d\mid\} :  k$. This
results in $\{d\mid\} :  i-\{d\mid\} :  i = 0$.

Suppose $k\leq 0$. Now, in $G$,
Right has moves in both components but, again, Left has the corresponding move in the other component. This leaves a position equal to $0$. Thus $\{d\mid\} :  k - \{d\mid\} :  (k-1)\geqslant 0$ for all $m$.

This result shows that
\begin{eqnarray*}
\{d\mid\} :  m &=& \{d, \{d\mid\} :  (m-1)\,\mid \,\},\text{ if $m>0$, and}\\
\{d\mid\} :  m &=& \{d\,\mid\, \{d\mid\} :  (m+1)\}, \text{ if $m<0$.}
\end{eqnarray*}
Finally, if $m>0$, then $d\leqslant \{d\mid\} :  (m-1)$. This follows since, in $ \{d\mid\} :  (m-1)-d\geqslant 0$, Right
can only move to $ \{d\mid\} :  (m-1)-d'$ where $-d'>-d$. Left responds to $d-d'>0$.

Thus, for $m>0$, the canonical form of $\{d\mid\} :  m$ is $\{ \{d\mid\} :  (m-1)\,\mid \}$.\\
\end{proof}
\begin{corollary}
Let  $d$ be a dyadic rational, $0\leq d <1$ and $m$ an integer. If  $m$ is  positive, then
$\{d\mid\} :  m=(\{d\mid\} :  m-1):1$. If $m$ is negative then $\{d\mid\} :  m=(\{d\mid \}:(m+1)):-1$.
\end{corollary}
\begin{proof}
If $m>0$, then
$$ (\{d\mid\} :  m-1):1 = \{\{d\mid\} :  m-1\mid \} = \{d\mid\} :  m.$$
If $m<0$, then
$$(\{d\mid \}:(m+1)):-1 = \{d\mid \{d\mid \}:(m+1)\} = \{d\mid\} :  m.
$$
\end{proof}

\begin{theorem}\label{thm:positive}
Let $d=_20.d_1d_2\ldots d_k$ and let $m$ be an integer.

\begin{enumerate}
\item If $m\geqslant 0$, then $\{d\mid \}:m=m+1$.
\item If $m<0$, then $\{d\mid \}:m=_20.d_1d_2d_3\ldots d_{j-1}1$,
where $j$ is the index of the $|m|$-th zero digit of the binary expansion of $d$.
\end{enumerate}

\end{theorem}

\begin{proof}
First suppose $m>0$. We have  $\{d\mid \}:m=(\{d\mid \}:m-1):1$. Since $\{d\mid \}:0=1$, then, by induction, $(\{d\mid \}:m-1):1=m:1$. Finally, $m:1=m+1$.\\

Now suppose $m<0$.

If $d=0$, then the theorem states $\{0\mid\}:m = 2^m$. This follows easily by induction as follows. First, $\{0\mid \}:0=1$ and,
$\{0\mid \}:m=\{0\mid \{0\mid \}:m+1\}$.  By induction, $\{0\mid \}:m=\{0\mid 2^{m+1}\}$,
and since, by \cref{simplicity1}, $\{0\mid 2^{m+1}\}= 2^m$ then this part of the result is proved.\\

We may now assume that $d>0$.

If $m=-1$, then $\{d\mid \}:-1=\{d\mid \{d\mid \}\}=\{d\mid 1\}$. Now, by Theorem \ref{simplicity1},
$\{d\mid 1 \}=_20.d_1d_2d_3\ldots d_{j-1}1$, where $j$ is the index of the first $0$-bit of the binary expansion of $d$.\\

If $m<-1$, then, by induction,  $\{d\mid\} :  (m+1)=_20.d_1d_2d_3\ldots d_{j-1}1$, where $j$ is the index of the
 $|m+1|$-th zero digit of the binary expansion of $d$. Now $\{d\mid\} :  m=\{d\mid \{d\mid\} :  (m+1)\}$.  Again, by Theorem \ref{simplicity1}, the binary expansion of $\{d\mid \{d\mid\} :  (m+1)\}$ is obtained by replacing by ``1'' the first $0$-bit in the binary expansion of $d$, after the position $j$ (and the following digits are all zero). That bit is the $|m|$-th zero digit of the binary expansion of $d$, and this finishes the proof.
\end{proof}

\begin{observation}
One consequence of \cref{thm:positive} is  that, for $0\leqslant d<1$ and $m$ a positive integer,  $\{\,d\mid\,\}:m=\{\,d\mid\,\}+m$, that is, the ordinal sum coincides with the usual sum.
\end{observation}

\begin{example}
\[\left\{\frac{309}{512}\,\Big|\,\right\}:-3=_2\{0.10011\textbf{\textbf{0}}101\,\mid\, \}:-3=_20.100111=\frac{39}{64}.\]
\end{example}

The signed binary notation is more useful for game practice because of, as mentioned before, the correspondence $1$-`blue edge' and $\overline{1}$-`red edge'. The following theorem, concerning the use of signed binary representations, is presented without proof, since it is similar to the previous one.\\

\begin{theorem}\label{negative3}
Let $d$ be a dyadic rational such that $0<d<1$, and $1.\overline{1}d_2\ldots d_k$ its signed binary expansion. Let $m$ be a negative integer. The signed binary expansion of $\{d\mid\,\}:m$ is obtained in the following way:\\

Case 1: If the number of minus ones in the signed binary expansion of $d$ is larger than $|m|$, then the signed binary expansion of $\{d\mid\,\}:m$ is $1.\overline{1}d_2d_3\ldots d_{i-1}$, where $i$ is the index of the $(|m|+1)$-th $\overline{1}$-bit in the signed binary expansion of $d$.\\

Case 2: If the number of minus ones in the signed binary expansion of $d$ ($n$) is less or equal than $|m|$, then the signed binary expansion of $\{d\,\mid\,\}:m$ is $1.\overline{1}d_2d_3\ldots d_{k}1\underbrace{\overline{1}\,\overline{1}\ldots\overline{1}\,\overline{1}}_{|m|-n\,\,\overline{1}'s}$.
\end{theorem}

\begin{example}
\begin{eqnarray*}
\left\{\frac{173}{512}\Big|\,\right\}:-3&=_2&\{1.\overline{1}\,\overline{1}1\overline{1}1
\overline{\textbf{\textbf{1}}}11\overline{1}\,\mid\,\}=_21.\overline{1}\,\overline{1}1\overline{1}1=\frac{11}{32}.\\
\end{eqnarray*}

\end{example}

%

The last case that needs to be evaluated is when  $G=\{d\,|\,\}:(m+d')$, $m$ is an integer and $d$ and $d'$ are dyadic rationals between $0$ and $1$.

\begin{corollary}\label{positive}
Let $d$ be a dyadic rational such that $0\leqslant d<1$.
\begin{enumerate}
\item If $m$ is a positive integer, then $\{d\mid\}:m=m+1$.
\item If $m$ be a negative integer, then $\{0\mid\}:m=2^m$.
\item If $m$ be a negative integer and $d=_20.d_1d_2\ldots d_k$, $d\ne 0$,
then $\{d\mid\}:m=_20.d_1d_2d_3\ldots d_{j-1}1$, where $j$ is the index of the $|m|$-th zero digit of the binary expansion of $d$.
\end{enumerate}
\end{corollary}

\begin{proof}
These are re-statements, via \cref{lem:translation}, of \cref{thm:positive} for part 1, and \cref{thm:positive}
for parts 2 and 3.
\end{proof}

\begin{theorem}[main result for numbers]\label{main2}
Consider $G=\{n+d\,|\,\}:(m+d')$ where $0\leqslant d,d'<1$ are dyadics, and $n,m\in\mathbb{Z}$. Let $\frac{k}{2^j}$ be the simplest form of $\{d\mid\}:m$. Then, $$G=n+\frac{k}{2^j}+\frac{d'}{2^j}.$$
\end{theorem}

\begin{proof}

By \cref{translation}, $\{n+d\,|\,\}:(m+d')=n+\left(\{d\,|\,\}:(m+d')\right)$, so we only need to analyze $G'=\{d\,|\,\}:(m+d')$, using the fact that $G=n+G'$.\\

Case 1: $d=0$ and $m\geqslant 0$.\\

We have that $G'$ is $\{0\,|\,\}:(m+d')$, and $\{0\,|\,\}$ is the canonical form of $1$. Therefore, by Theorem~\ref{roode},
  $G'=\{0\,|\,\}:(m+d')=1+m+d'$. By \cref{positive}, $\{d\mid\}:m=m+1=\frac{m+1}{2^0}$, $G'=\frac{m+1}{2^0}+\frac{d'}{2^0}$, and the theorem holds.\\

Case 2: $d=0$ and $m<0$.\\

We have that $G'$ is $\{0\,|\,\}:(m+d')$, and $\{0\,|\,\}$ is the canonical form of $1$. Therefore, by Theorem~\ref{roode},
  $$G'=\{0\,|\,\}:(m+d')=1+\frac{1}{2^{|m|}}(1-2^{|m|}+d')=\frac{1}{2^{|m|}}+\frac{d'}{2^{|m|}}.$$ By \cref{positive}, $\{d\mid\}:m=\frac{1}{2^{|m|}}$, $G'=\frac{1}{2^{|m|}}+\frac{d'}{2^{|m|}}$, and the theorem holds.\\

Case 3: $d>0$ and $m\geqslant 0$.\\

By \cref{positive} part 1, $G'=m+d'+1$. By \cref{positive} part 2, $\{d\mid\}:m=\frac{m+1}{2^{0}}$, $G'=\frac{m+1}{2^{0}}+\frac{d'}{2^{0}}$, and the theorem holds.\\

Case 4: $d>0$ and $m<0$.\\

This is the hardest case. In order to prove it we will construct a {\sc blue-red hackenbush} string $H$ whose value is $\frac{k}{2^j}+\frac{d'}{2^j}$ (Part 1). We will then prove that $G'-H$ is a $\mathcal{P}$-position (Part 2).\\

(Part 1) Let $\underbrace{1.\overline{1}1\ldots1\overline{1}1\ldots1\overline{1}\ldots\overline{1}1\ldots1}_{n\,\,\overline{1}'s}$ be the signed binary expansion of $d$. Since $0<d<1$, the first digit after the binary point is $\overline{1}$. Also, assume that this expansion has $n$ $\overline{1}$'s.\\

Let $1.\overline{1}d'_2d'_3\ldots d'_{w}$ be the signed binary expansion of $d'$. Since $0<d<1$, the first digit after the binary point is $\overline{1}$.\\

Consider the hardest case $|m|>n$. By Theorem \ref{negative3}, we know that $$\underbrace{1.\overline{1}1\ldots1\overline{1}1\ldots1\overline{1}\ldots\overline{1}1\ldots1}_{d\,(n\,\overline{1}'s)}1\underbrace{\overline{1}\,\overline{1}\ldots\overline{1}\,\overline{1}}_{|m|-n\,\,\overline{1}'s}$$ is the signed binary expansion of the game value of $\{d\mid\}:m$. The hypothesis of the current theorem states that this is $\frac{k}{2^j}$. Hence, there are $j$ binary places.\\

Now, the game value of the following {\sc blue-red hackenbush} string $H$ is $\frac{k}{2^j}+\frac{d'}{2^j}$. That happens because the added rightmost part $d'$ is shifted by $j$ binary places.

$$\underbrace{\underbrace{1.\overline{1}1\ldots1\overline{1}1\ldots1\overline{1}\ldots\overline{1}1\ldots1}_{d\,(n\,\overline{1}'s)}1\underbrace{\overline{1}\,\overline{1}\ldots\overline{1}\,\overline{1}}_{|m|-n\,\,\overline{1}'s}}_{\{d\mid\}:m=\frac{k}{2^j}\,\,(|m|\,\overline{1}'s)}\underbrace{1\overline{1}d'_2d'_3\ldots d'_{w}}_{shape\,of\,d'}=\frac{k}{2^j}+\frac{d'}{2^j}$$

(Part 2) In order to finish the proof, we have to show that $G'-H$ is a $\mathcal{P}$-position. By Theorem \ref{th:g}, we can use the following game form of $G'$, which also uses {\sc blue-red hackenbush} strings. The subordinate is a {\sc blue-red hackenbush} string whose value is $m+d'$.

$$G'=\{\underbrace{1.\overline{1}1\ldots1\overline{1}1\ldots1\overline{1}\ldots\overline{1}1\ldots1}_{d\,(n\,\overline{1}'s)}\,|\,\}:\underbrace{\overline{1}\ldots\overline{1}}_{(|m|\,\overline{1}'s)}\underbrace{1\overline{1}d'_2d'_3\ldots d'_{w}}_{shape\,of\,d'} $$

Let us verify that $G'-H=0$, that is, let us check that
$$\{\underbrace{1.\overline{1}1\ldots1\overline{1}1\ldots1\overline{1}\ldots\overline{1}1\ldots1}_{d\,(n\,\overline{1}'s)}\,|\,\}:\underbrace{\overline{1}\ldots\overline{1}}_{(|m|\,\overline{1}'s)}\underbrace{1\overline{1}d'_2d'_3\ldots d'_{w}}_{shape\,of\,d'}$$

\vspace{-0.8cm}
$$+$$

\vspace{-0.6cm}
$$\underbrace{\underbrace{\overline{1}.1\overline{1}\ldots\overline{1}1\overline{1}\ldots\overline{1}1\ldots1\overline{1}\ldots\overline{1}}_{-d\,(n\,1's)}\overline{1}\underbrace{1\,1\ldots1\,1}_{|m|-n\,\,1's}}_{\{\mid-d\}:(-m)\,\,(|m|\,1's)}\underbrace{\overline{1}1\overline{d'_2}\,\overline{d'_3}\ldots \overline{d'_{w}}}_{shape\,of\,-d'}$$
is a $\mathcal{P}$-position.\\

First, there is a correspondence between the moves in the shapes of $d'$ and $-d'$. Also, there is a correspondence between Right moves in the $|m|$ $\overline{1}$'s of the subordinate of the upper component and Left moves in the ones of\linebreak $\{\mid-d\}:(-m)$ in the bottom component. Regarding those correspondences, there is a Tweedledee-Tweedledum strategy.\\

Second, if Left moves to $1.\overline{1}1\ldots1\overline{1}1\ldots1\overline{1}\ldots\overline{1}1\ldots1=d$ in the upper component (entering the base), Right answers by removing the $\overline{1}$ immediately after the shape of $-d$ in the bottom component, and vice-versa.\\

Third, if Right removes any $\overline{1}$ of the shape of $-d$ in the bottom component, Left answers with $1.\overline{1}1\ldots1\overline{1}1\ldots1\overline{1}\ldots\overline{1}1\ldots1$ (entering the base) in the upper component, and wins.\\

Since the second player wins, $G'-H\in\mathcal{P}$, and $G'=H=\frac{k}{2^j}+\frac{d'}{2^j}$.
\end{proof}

\vspace{0.3cm}
\begin{observation}
Essentially, if $m+d'\geqslant 0$, the ordinal sum $\{d\,|\,\}:(m+d')$ is the sum $\{d\,|\,\}+m+d'$; if, instead,  $m+d'< 0$, Corollary \ref{positive} is needed.
\end{observation}

\subsection{Determination of the game value of a \sc{clockwise}\\ \sc{blue-red hackenbush} position}
Consider again the \textsc{clockwise blue-red hackenbush} position exhibited in Figure \ref{fig:trunk1}. In order to compute its game value, let us compute first the game value of the subposition presented in Figure \ref{fig: subposition}.

\begin{figure}[hbt]
\begin{center}
\includegraphics[scale=0.9]{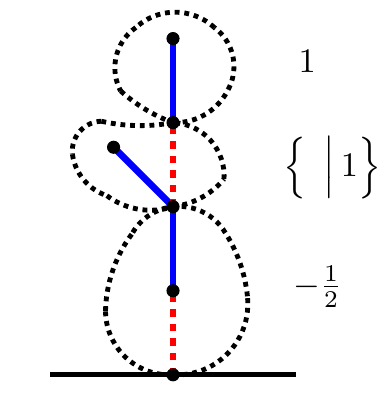}
\caption{A relevant subposition.}\label{fig: subposition}
\end{center}
\end{figure}

We have to determine the value of $-\frac{1}{2}:\left(\{\,|\,1\}:1\right)$. In order to compute $\{\,|\,1\}:1$, we need to position to be in the correct form to apply Theorem~\ref{main2}. Hence, we instead use $\{-1\,|\,\}:-1$ and will negate the resulting value. 

Consider $\left(\{-1\,|\,\}:-1\right)$. By Theorem \ref{main2}, since $n=-1$, $d=0$, $m=-1$, $d'=0$, and $\{d\,|\,\}:m=\frac{1}{2}$, we have $\{-1\,|\,\}:-1=-1+\frac{1}{2}=-\frac{1}{2}$. Hence, $\{\,|\,1\}:1=\frac{1}{2}$.\\

Finally, using van Roode's evaluation, $-\frac{1}{2}:\left(\{\,|\,1\}:1\right)=-\frac{1}{2}:\frac{1}{2}=-\frac{3}{8}$.\\

Regarding the \textsc{clockwise blue-red hackenbush} position exhibited in Figure \ref{fig:trunk1}, we have the situation presented in Figure \ref{fig: position}.\\

\begin{figure}[hbt]
\begin{center}
\includegraphics[scale=0.9]{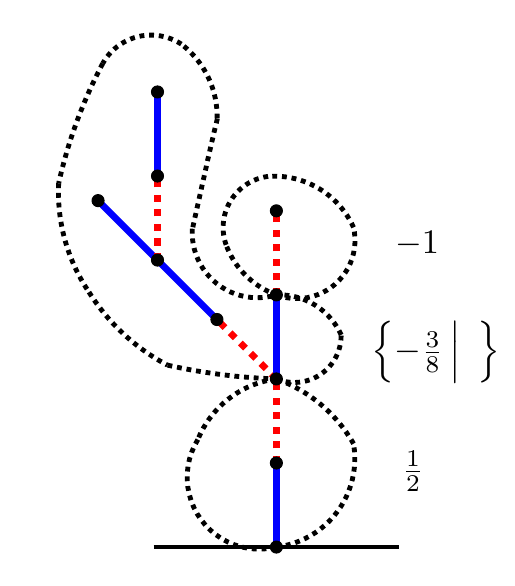}
\caption{Figure \ref{fig:trunk1} revisited.}\label{fig: position}
\end{center}
\end{figure}

We have to determine the value of $\frac{1}{2}:\left(\{-\frac{3}{8}\,\,|\,\}:-1\right)$; we start with $\{-\frac{3}{8}\,\,|\,\}:-1$. To apply Theorem~\ref{main2}, to find the value of $\{-\frac{3}{8}\,\,|\,\}:-1$, we first rewrite the expression as $\left\{-1+(\frac{5}{8})\,\,\big|\,\right\}:-1$. We observe that $n=-1$, $d=\frac{5}{8}$, $m=-1$, $d'=0$. By Theorem \ref{thm:positive}, $$\{d\,|\,\}:m=_2\{0.1\textbf{0}1\,\mid\, \}:-1=_20.11=\frac{3}{4}.$$ Now using Theorem~\ref{main2}, we have $\{-\frac{3}{8}\,\,|\,\}:-1=-1+\frac{3}{4}=-\frac{1}{4}$.\\

Using again van Roode's evaluation, $\frac{1}{2}:\left(\{-\frac{3}{8}\,\,|\,\}:-1\right)=\frac{1}{2}:-\frac{1}{4}=\frac{7}{16}$. This is the game value of the proposed position.\\

\noindent
\textbf{Exercise:} Verify that the game value of the \textsc{clockwise blue-red hackenbush} position exhibited in Figure \ref{fig:trunk2} is given by $$\{-2\,|\,\}:\left(-1:\left(\{-1\,|\,\}:-1\right)\right)=-1\frac{1}{4}.$$

\section{{\sc domino shave}}\label{sec:DominoShave}
We first find a normalized version of \textsc{domino shave} and then show that
this is equivalent to \textsc{clockwise hackenbush} by giving a bijection between the positions. We then
note which selection of dominoes give rise to games already in the literature. As well, we show that Hetyei's Bernoulli game is a subset of \textsc{stirling shave}. 

\subsection{\textit{Normalized} {\sc domino shave}}

Let $D$ be a \textsc{domino shave} position  $d_1,d_2,\ldots,d_k$. We \textit{normalize} the string using the following algorithm.\\

\begin{enumerate}
\item  Set $s=1$ and $p=1$.
\item  In the right-most consecutive line of blue, green and red dominoes, let $E_s$ be the set of  indices of the dominoes that can be played. 

  Consider the dominoes with indices in $E_s$. Starting at the left (least index) domino:
  \begin{itemize}
  \item if it is blue, then replace it by $(p,p+1)$ coloured aqua;
  \item if it is red, replace it by $(p+1,p)$, coloured pink;
\item  if it is green then replace it by $(p,p)$, coloured emerald.
\end{itemize}
Repeat with the  blue, red or green domino of least index in $E_s$. When all dominoes in $E_s$ have been replaced go to step (3).
\item Set $s:=s+1$ and $p:=p+2$.  If there are any blue, red or green dominoes, repeat step 1. If not then recolour the aqua dominoes blue, the pink dominoes red, and the emerald dominoes green and stop.
\end{enumerate}

\begin{example} Let $G=(2,4)(7,3)(1,2)(4,4)(3,2)$. The steps of the algorithm are shown in Table~\ref{tbl: conversion}, where a change of colour is indicated by $[a,b]$.
\begin{table}[ht]
\caption{Conversion to Normalized \textsc{domino shave}}\label{tbl: conversion}
\begin{center}
\begin{tabular}{|c|c|c|c|}\hline
$(s,p)$&Old Line&$E_s$&New Line\\
\hline\hline
$(1,1)$&$(2,4)(7,3)(1,2)(4,4)(3,2)$&$(1,2)(3,2)$&$(2,4)(7,3)[1,2](4,4)[2,1]$\\ \hline
$(2,3)$&$(2,4)(7,3)[1,2](4,4)[2,1]$&$(4,4)$&$(2,4)(7,3)[1,2][3,3][2,1]$\\ \hline
$(3,5)$&$(2,4)(7,3)[1,2][3,3][2,1]$&$(2,4)(7,3)$&$[5,6][6,5][1,2][3,3][2,1]$\\ \hline
$(4,7)$&$[5,6][6,5][1,2][3,3][2,1]$&&$(5,6)(6,5)(1,2)(3,3)(2,1)$\\ \hline
\end{tabular}
\end{center}
\label{default}
\end{table}%
\end{example}

The partition of the indices into $E_1, E_2,\ldots$ is independent of the normalization. It does point to a very important result.

\begin{lemma} Let $f$ be the largest index in $E_a$, $a>1$. The domino $d_{f+1}$ prevents every domino in $E_a$ from being played.
\end{lemma}
\begin{proof} Let $g$ be the smallest index of the dominoes in $E_a$. This gives $E_a=\{g, g+1, \ldots,f\}$.
 After $d_{f+1}$ has been played then every domino $d_i$, $g\leqslant i\leqslant f$ is playable. Thus
 \[\min\{l_f,r_f\}\geqslant \min\{l_{f-1},r_{f-1}\}\geqslant \ldots\geqslant \min\{l_g,r_g\}. \]
 Since $g\in E_a$, there exists $d_j$, $j>g$, and $j\in E_{a-1}$ which prevents $d_g$ from being played. (If no such domino exists then $g\in E_{a-1}$.)
 We may assume that $j$ is the least index. Thus $\min\{l_g,r_g\} >\min\{l_j,r_j\}$. Since  $f+1,j\in E_{a-1}$ and $f+1\leqslant j$ then
 $d_j$ does not prevent $d_{f+1}$ being played. This gives $\min\{l_j,r_j\} >\min\{l_{f+1},r_{f+1}\}$. Combining the inequalities yields
$ \min\{l_i,r_i\}>\min\{l_{f+1},r_{f+1}\}$ for $g\leqslant i\leqslant f$. That is, $d_{f+1}$ prevents all of $E_a$ being played.
\end{proof}

The properties of the normalization algorithm that we require follow immediately from the algorithm steps.

\begin{lemma}\label{lem:algorithmproperties} Let $D=(d_1,d_2,\ldots,d_k)$ be a \textsc{domino shave} position and $D'=(d'_1,d'_2,\ldots,d'_k)$ be the normalized position.
\begin{enumerate}
\item The indices of the dominoes of $D$ are partitioned into subsets $E_1,E_2,\ldots,E_f$;
\item If $i\in E_a$, $j\in E_b$ and $a<b$ then the left and right spots of $d'_i$ are smaller than the left and right spots of $d'_j$.
\item Let $i<j$, $i\in E_a$ and $j\in E_b$. If $a<b$ then
$d'_j$ does not prevent $d'_i$ being played. If $a>b$ then $d'_j$ does prevent $d'_i$ being played.
\end{enumerate}
\end{lemma}

\begin{lemma} If $D$ is a \textsc{domino shave} position and $D'$  is its normalized version then $D=D'$.
\end{lemma}
\begin{proof} Let $\{d_i: i=1,2,\ldots, k\}$ be the dominoes in $D$ and $\{d'_i: i=1,2,\ldots, k\}$ the dominoes in $D'$. We show that $D-D'=0$.

The strategy will the usual mimic strategy: if the first player plays the domino with index $i$ in one of the two strings then the second player plays the other domino of index $i$. To prove this we need to show that at every stage of the game,  $d_i$ is playable  if and only if $d'_i$ is playable.\\

On the first move, the only dominoes playable in $D$ are those $d_i$, $i\in E_1$. By \cref{lem:algorithmproperties} ($3$), the dominoes $d'_i$, $i\in E_a$, $a>1$ are not playable.

Now consider the dominoes $d'_i, d'_j$,  $i,j\in E_1$, and $i<j$.  If there is a green
domino $d'_c$, $c\in E_1$, $i<c\leqslant j$  then both spots of  $d'_j$ are greater than those of $d'_i$. If there is no such domino then the spots of $d'_i$ and $d'_j$ are $p$ and $p+1$ for some $p$. The order depends on the domino colour.  Consequently,  $d'_j$ does not prevent the playing of $d'_i$. Therefore, for $i\in E_1$, both $d_i$ and $d'_i$ are playable.\\

Now suppose  $d_i$, $i\in E_a$, $a>1$, is playable. This is only possible if $d_j$, $j>i$, $j\in E_b$, and $b<a$ have been played or eliminated. By the mimic strategy played so far,
it is also true that $d'_j$, $j>i$, $j\in E_b$, and $b<a$ have been played or eliminated.  By \cref{lem:algorithmproperties} ($3$),  the dominoes $d'_i$, $i\in E_b$, $b>a$ are not playable. By \cref{lem:algorithmproperties} ($2$), the dominoes $d'_j$, $i<j$ and $i,j\in E_b$, do not prevent $d'_i$ from being played. Therefore $d'_i$ is playable.

Suppose  $d'_i$, $i\in E_a$, is playable. Again, the dominoes  $d'_j$, $j>i$, $j\in E_b$, and $b<a$ have been played or eliminated.
By the mimic strategy played so far,
it is also true that $d_j$, $j>i$, $j\in E_b$, and $b<a$ have been played or eliminated. However, by the normalization algorithm, once the dominoes
$d_j$, $j\in \cup_{f=1}^{a-1}$, and $i<j$ are gone then every domino, and specifically $d_i$,  with index in $E_a$ is playable.

This shows that the mimic strategy is possible and therefore $D-D'$ is a second player win.
\end{proof}



\subsection{{\sc domino shave} is {\sc clockwise hackenbush}}\label{sec:DShave}

 The proof of the equivalence between \textsc{domino shave} and  \textsc{clockwise hackenbush} is similar to that of  \textsc{domino shave} and normalized  \textsc{domino shave}.
\begin{theorem} There is a bijection, $f$,  between \textsc{domino shave} and  \textsc{clockwise hackenbush} positions such that $D-f(D) = 0$.
\end{theorem}

\begin{proof}
Let $D=(d_1,d_2,\ldots,d_k)$ be a normalized  \textsc{domino shave} position. Let $E_a$ be the index set of the last line of dominoes replaced in the normalization algorithm. We induct on $a$.

Suppose $a=1$. Let $T=(e_1,e_2,\ldots,e_k)$ be a string where $e_i$ is the same colour as $d_i$. Every domino in $D$ is playable and remains playable until it is eliminated. Similarly, $T$ is a trunk so every edge is playable and remains playable until it is removed.

Suppose $a>1$. Consider $D'=D\setminus \{d_i:i\in E_a\}$. Now $D'$ is a normalized  \textsc{domino shave} position and by induction, there exists a
unique \textsc{clockwise hackenbush} $T'$ with $f(D')=T'$. Also, $D''=(d_i:i\in E_a)$ is equivalent to a string $T''$. Let $j$ be the greatest index, in $E_a$
and let $e_{j+1}$ be the edge of $T'$ which corresponds to $d_{j+1}$.
Create a new tree, $T$, by identifying the bottom vertex of $T''$ and the bottom vertex of $e_{j+1}$. Place $T''$ to the left of the edge $e_{j+1}$. Set $f(D) = T$. Note that every edge of $T$ is associated with a domino, specifically, $d_i\leftrightarrow e_i$.\\

\textit{Claim:} $D-T=0$.

\textit{Proof of Claim.} This follows in a similar fashion the previous equivalence result. The mimic strategy is to play the corresponding other object of the same index.

If $a=1$ then all edges and dominoes are playable and remain playable until eliminated.

If $i\not\in E_a$ then both the dominoes in $D''$ and the edges of $T''$ do not prevent  $d_i$ and $e_i$ from being played.

Suppose $i\in E_a$. If $d_i$ is playable then $d_{j+1}$ has been eliminated. In $T$, therefore, $e_{j+1}$ has also been eliminated. The string $T''$ is now part of the trunk and every edge, including $e_i$ is playable. If $e_i$ is playable then it is on the trunk and $e_{j+1}$ has been eliminated. Therefore, $d_{j+1}$ has been eliminated and
every domino in $D''$, including $d_i$ is playable.\\

This proves the Claim and the equivalence.
 \end{proof}

From a \textsc{clockwise hackenbush} it is possible to get the normalized \textsc{domino shave} position by realizing the first trunk corresponds to the dominoes in $E_1$ and the next strings to the left, in order, correspond to the dominoes of $E_2, E_3,\ldots, E_n$ . The normalization algorithm then gives a set of dominoes.
\\

\subsection{Relationship with other games}

Versions of \textsc{domino shave} include, as special cases, several other rulesets each of which has been shown to have interesting or intriguing properties.
\begin{enumerate}
\item If all the dominoes  are $(1,1)$ then the \textsc{clockwise hackenbush} version is a single string of green edges.  This is {\sc nim}, which  is the foundation of all impartial games \cite{Bouto1902}.
\item If all the pieces are of the form $(a,a)$ then this is {\sc stirling shave} \cite{Fisher}. An explicit formula for evaluating the ordinal sums of nimbers is developed to give the values of the positions. If the dominoes are a permutation of the dominoes $(1,1),(2,2),\ldots,(n,n)$ then the number of $\P$-positions of length $n$ is given in terms of the Stirling numbers of the second kind.
\item Hetyei \cite{Hetye2009} did not give the game a name. The domino $d_i$ is restricted to having both spots between $1$ and $i$.
Only the right spot is used to determine when a domino can be removed thus it is an impartial game.
The number of $\P$-positions with $n$ dominoes is given in terms of the Bernoulli numbers of the second kind. The game can be shown to be equivalent to
\textsc{stirling shave} via the following. A domino is unplayable if it can never be the first of the string to be removed. A blue domino is unplayable
since the right stop is greater than the left. If the dominoes $d_{i+1},d_{i+2},\ldots, d_j$ are unplayable
and $d_{i}$ is prevented from being played by $d_f$, $i+1\leqslant f\leqslant j$ then $d_i$ is unplayable. Removing all unplayable vertices
does not affect the options of all followers in the game. Remaining are a subset of the red and green dominoes all of which have their right spots no larger than their left spots.
 If $d_i$ is prevented from being played by $d_j$ then, in particular, $r_i>r_j$. Thus when each domino $(l,r)$ is replaced by $(r,r)$, the same dominoes can be played and the dominoes that prevent a domino from being played is the same in both games.  \item If all the pieces are $(1,2)$ and $(2,1)$ then it is equivalent to  \textsc{blue-red hackenbush} strings.
\item If all the pieces are $(1,2)$, $(2,1)$ and $(1,1)$ then this is  \textsc{blue-red-green hackenbush} strings. The value can be given via the ordinal sums of numbers and nimbers. No one has given an explicit formula. It seems clear that the values of the strings are unique but we do not know of a proof.
\end{enumerate}

\section*{Acknowledgments}
Alda Carvalho is a CEMAPRE member and has the support of
Project CEMAPRE - UID/MULTI/00491/2019 financed by FCT/MCTES through national funds.\\

\noindent
Melissa A. Huggan was supported by the Natural Sciences and Engineering Research Council of Canada (funding reference number PDF-532564-2019).\\

\noindent
Richard J. Nowakowski was supported by the Natural Sciences and Engineering Research Council of Canada (funding reference number 4139-2014).\\

\noindent
Carlos Santos is a CEAFEL member and has the support of
UID/MAT/04721/2019 strategic project.\\

\end{document}